\documentclass[11pt, reqno]{amsart}
\usepackage[T1]{fontenc}
\usepackage[utf8]{inputenc}
\usepackage{babel}
\usepackage{amstext}
\usepackage{amsthm}
\usepackage{xcolor}
\usepackage{amssymb}
\usepackage{enumitem}
\usepackage{setspace}
\usepackage{mathrsfs}
\usepackage[hidelinks,colorlinks=true,linkcolor=blue,citecolor=blue]{hyperref}
\usepackage[colorinlistoftodos,prependcaption,textsize=tiny]{todonotes}
\usepackage{lipsum}   
\newtheorem*{assumptions*}{Assumptions}

\usepackage{xargs}                   
\usepackage{xargs}   
\usepackage{hyperref}
\renewcommand\eqref[1]{(\ref{#1})}

\makeatletter
\newcommand*{\mint}[1]{%
  \mint@l{#1}{}%
}
\newcommand*{\mint@l}[2]{%
  \@ifnextchar\limits{%
    \mint@l{#1}%
  }{%
    \@ifnextchar\nolimits{%
      \mint@l{#1}%
    }{%
      \@ifnextchar\displaylimits{%
        \mint@l{#1}%
      }{%
        \mint@s{#2}{#1}%
      }%
    }%
  }%
}
\newcommand*{\mint@s}[2]{%
  \@ifnextchar_{%
    \mint@sub{#1}{#2}%
  }{%
    \@ifnextchar^{%
      \mint@sup{#1}{#2}%
    }{%
      \mint@{#1}{#2}{}{}%
    }%
  }%
}
\def\mint@sub#1#2_#3{%
  \@ifnextchar^{%
    \mint@sub@sup{#1}{#2}{#3}%
  }{%
    \mint@{#1}{#2}{#3}{}%
  }%
}
\def\mint@sup#1#2^#3{%
  \@ifnextchar_{%
    \mint@sup@sub{#1}{#2}{#3}%
  }{%
    \mint@{#1}{#2}{}{#3}%
  }%
}
\def\mint@sub@sup#1#2#3^#4{%
  \mint@{#1}{#2}{#3}{#4}%
}
\def\mint@sup@sub#1#2#3_#4{%
  \mint@{#1}{#2}{#4}{#3}%
}
\newcommand*{\mint@}[4]{%
  \mathop{}%
  \mkern-\thinmuskip
  \mathchoice{%
    \mint@@{#1}{#2}{#3}{#4}%
        \displaystyle\textstyle\scriptstyle
  }{%
    \mint@@{#1}{#2}{#3}{#4}%
        \textstyle\scriptstyle\scriptstyle
  }{%
    \mint@@{#1}{#2}{#3}{#4}%
        \scriptstyle\scriptscriptstyle\scriptscriptstyle
  }{%
    \mint@@{#1}{#2}{#3}{#4}%
        \scriptscriptstyle\scriptscriptstyle\scriptscriptstyle
  }%
  \mkern-\thinmuskip
  \int#1%
  \ifx\\#3\\\else_{#3}\fi
  \ifx\\#4\\\else^{#4}\fi
}
\newcommand*{\mint@@}[7]{%
  \begingroup
    \sbox0{$#5\int\m@th$}%
    \sbox2{$#5\int_{}\m@th$}%
    \dimen2=\wd0 %
    \let\mint@limits=#1\relax
    \ifx\mint@limits\relax
      \sbox4{$#5\int_{\kern1sp}^{\kern1sp}\m@th$}%
      \ifdim\wd4>\wd2 %
        \let\mint@limits=\nolimits
      \else
        \let\mint@limits=\limits
      \fi
    \fi
    \ifx\mint@limits\displaylimits
      \ifx#5\displaystyle
        \let\mint@limits=\limits
      \fi
    \fi
    \ifx\mint@limits\limits
      \sbox0{$#7#3\m@th$}%
      \sbox2{$#7#4\m@th$}%
      \ifdim\wd0>\dimen2 %
        \dimen2=\wd0 %
      \fi
      \ifdim\wd2>\dimen2 %
        \dimen2=\wd2 %
      \fi
    \fi
    \rlap{%
      $#5%
        \vcenter{%
          \hbox to\dimen2{%
            \hss
            $#6{#2}\m@th$%
            \hss
          }%
        }%
      $%
    }%
  \endgroup
}
\numberwithin{equation}{section}
\theoremstyle{plain}
\newtheorem{thm}{Theorem}[section]

\newtheorem{cor}[thm]{Corollary}
\newtheorem{lem}[thm]{Lemma}

\theoremstyle{definition}
\newtheorem{defn}[thm]{Definition}
\newtheorem{rem}[thm]{Remark}
\newtheorem{ex}[thm]{Example}
\newtheorem{as}[thm]{Assumption}
\newcommand{\G}{\mathbb{G}}

\def\G{\mathbb{G}}

\def\L{\mathcal{L}_{p}}

\def\HH{\mathcal{H}}
\def\G{{\mathbb G}}
\def\L{\mathcal{L}}

\def\Rn{\mathbb{R}^{n}}
\def\Rp{\mathbb{R}^{p}}
\def\RN{\mathbb{R}^{N}}

\def\Ga{\Gamma}

\title[Fujita exponent for heat equation with H\"{o}rmander vector fields]{Fujita exponent for heat equation with H\"{o}rmander vector fields}
\author[M. Chatzakou]{Marianna Chatzakou}
\address{
	Marianna Chatzakou:
	\endgraf
    Department of Mathematics: Analysis, Logic and Discrete Mathematics
    \endgraf
    Ghent University, Belgium
  	\endgraf
	{\it E-mail address} {\rm marianna.chatzakou@ugent.be}
		}
	
\author[A. Kassymov]{Aidyn Kassymov}
\address{
  Aidyn Kassymov:
  \endgraf
  Institute of Mathematics and Mathematical Modeling
  \endgraf
  28 Shevchenko str.
  \endgraf
  050010 Almaty
  \endgraf
  Kazakhstan
  \endgraf
	{\it E-mail address}  {\rm kassymov@math.kz}}

\author[M. Ruzhansky]{Michael Ruzhansky}
\address{
  Michael Ruzhansky:
  \endgraf
  Department of Mathematics: Analysis, Logic and Discrete Mathematics
  \endgraf
  Ghent University, Belgium
  \endgraf
 and
  \endgraf
  School of Mathematical Sciences
  \endgraf
  Queen Mary University of London
  \endgraf
  United Kingdom
  \endgraf
  {\it E-mail address} {\rm michael.ruzhansky@ugent.be}
  }
	
\newcommand{\vertiii}[1]{{\left\vert\kern-0.25ex\left\vert\kern-0.25ex\left\vert #1 
    \right\vert\kern-0.25ex\right\vert\kern-0.25ex\right\vert}}
\allowdisplaybreaks

\begin{document}

\thanks{The authors are supported by the FWO Odysseus 1 grant G.0H94.18N: Analysis and Partial Differential Equations and by the Methusalem programme of the Ghent University Special Research Fund (BOF) (Grant number 01M01021). Marianna Chatzakou is a postdoctoral fellow of
the Research Foundation – Flanders (FWO) under the postdoctoral grant No 12B1223N. The research of Michael Ruzhansky and Aidyn Kassymov is also supported by FWO grant G083525N and the MHES of the Republic Kazakhstan grant (AP23484106), respectively. 
{\it Acknowledgments:}
The authors would like to thank  Dr. Berikbol T. Torebek for comments and suggestions.}

\begin{abstract} 
In this paper, we show global existence and non-existence results for the heat equation with some of the squares of  smooth vector fields on $\Rn$ satisfying H\"{o}rmander's rank condition with a non-linearity of the form $f(u)$, where $f$ is a suitable function and $u$ is the solution. In particular, when $f(u)=u^p$, we calculate the critical Fujita exponent.  We also give necessary conditions for blow-up or, alternatively, a sufficient condition for the existence of positive global solutions for time-dependent nonlinearities of the type $\varphi(t)f(u)$.
\end{abstract} 

\maketitle

\tableofcontents

\section{Introduction}
In 1966, Fujita in \cite{F66} proved that the initial value problem for the reaction-diffusion equation
\begin{equation}\label{Heat_Cauchyrn}
\begin{split}
& u_{t}(t,x)-\Delta u(t,x) =u^p(t,x),\,\, x\in \mathbb{R}^N,\, t>0, \\&
u(0,x)=u_0(x)\geq 0,\, x\in \mathbb{R}^N,
   \end{split}
\end{equation}
where $u_0\not\equiv 0$ has no global in time solution for $p$ in the range $1<p<p_F=1+\frac{2}{N}$, where  $p_F$ is called the \textit{critical Fujita exponent}, or just the \textit{Fujita exponent}. Later, it was shown in \cite{H73} and \cite{AW78} that this result holds for $p = p_F$ as well. 
Moreover, Weissler in \cite{Weissler} extended Fujita's previous result and proved that if $p>p_{F}$, $0\leq u_{0}\in L^{\gamma}(\Rn)$ with $\gamma\geq1$, and 
$$\int_{0}^{\infty}\|\text{e}^{-s\Delta}u_{0}\|_{L^{\infty}(\Rn)}^{p-1}ds<\frac{1}{p-1},$$
then there exists a non-negative global solution of \eqref{Heat_Cauchyrn}.

The Fujita phenomenon has  been observed in a broad range of equations with different types of non-linearities, on various geometries, and under assorted boundary conditions.  Despite the abundance of extensions and refinements of this problem, we do not intend to summarise them here; many open questions remain—particularly regarding, roughly speaking, the criticality, i.e., the study of the critical exponent $p_{F}$, in more general settings and types of non-linearities. For a recent exposition of the related literature we may refer to e.g. the work of Quittner and Souplet \cite{Soup}, and to the paper  \cite{Lev90}. 

Let us mention here some works on the topic which are particularly relevant to our analysis; the Fujita exponent for the heat equation in the case of an unbounded domain was considered in the sub-Riemannian setting, cf. \cite{Ruzhansky}. In the latter case, and more particularly, when the sub-Riemannian underlying manifold is a unimodular Lie group, the global well--posedeness depends on the volume growth, and their approach used heat kernel estimates. For the particular case of the Heisenberg groups $\mathbb{H}^n$  the problem was studied by several authors, \cite{Kirane 1, Kirane 2, BRT22, Birindelli, D'Ambrosio, D'Ambrosio1, CKR24}.  In the work \cite{CKR24} the authors considered the initial heat problem with the addition of a time dependent non-linearity term and showed the identification of the necessary and sufficient conditions for the global well-posedeness of the corresponding heat equation on $\mathbb{H}^1$ exactly as in it happens in the Euclidean case \cite{CH24}. 

In this work we study the Fujita problem on $\Rn$ where the non-linearity is not restricted to the more classical type of the form $f(u)=u^p$ and the differential operator $\L$ is not restricted to the elliptic operator $\Delta$, but arises via a system of vector fields $X=\{X_1,\ldots,X_m\}$ on $\Rn$   satisfying Assumption \ref{as1}, i.e., we consider the non-linear Cauchy problem of the form 
\begin{equation}
    \label{intro.prob}
    \begin{cases}
         u_{t}(t,x)-\L u(t,x)= f(u(t,x)),\,\,\,\,\,(t,x)\in \mathbb{R}_{+}\times \Rn,\\
          u(0,x)=u_{0}(x),\,\,\,\,\,x\in \Rn,
    \end{cases}
\end{equation}
where $\L=\sum\limits_{i=1}^{m}X_{i}^{2}$ is a H\"{o}rmander sum of squares.  Our main results can be summarised as follows:
\begin{thm}\label{THM.main}
Let $X$ be a system of $m$ vector fields on $\Rn$ satisfying Assumption \ref{as1}.

(i) Let $1<\alpha<\alpha_F=1+\frac{2}{q}$ where $q$ is as in \eqref{homdim}. Suppose that $f:[0,\infty)\rightarrow \mathbb{R}$ is a locally integrable function such that $f(u)\geq Bu^{\alpha}$ for some $B>0$. Let $0< u_{0}$ be a measurable function on $\Rn$. Then the differential inequality 
\begin{equation}\label{inEQ:main/thm}
    \begin{cases}
        u_{t}-\L u
        \geq f(u),\,\,\,\,\,(t,x)\in (0,+\infty)\times \Rn,\\ u(0,x)=u_{0}(x),\,\,\,\,x\in\Rn,
    \end{cases}
\end{equation}
    does not admit a (non-trivial) distributional solution $u \geq 0$ in $(0,\infty)\times \Rn$.
    
    (ii) Let $\alpha=\alpha_F=1+\frac{2}{q}$ where $q$ is as in \eqref{homdim}. Suppose that $f:[0,\infty)\rightarrow \mathbb{R}$ is a locally integrable function such that $f(u)\geq Bu^{\alpha}$ for some $B>0$. Let $0< u_{0}$ be a measurable function on $\Rn$. Then the differential equality 
\begin{equation*}
    \begin{cases}
        u_{t}-\L u=f(u),\,\,\,\,\,(t,x)\in (0,+\infty)\times \Rn,\\ u(0,x)=u_{0}(x),\,\,\,\,x\in\Rn,
    \end{cases}
\end{equation*}
    does not admit a (non-trivial) distributional solution $u \geq 0$ in $(0,+\infty)\times \Rn$.

    (iii) Let $\alpha>\alpha_F=1+\frac{2}{q}$ where $q$ is as in \eqref{homdim}. Suppose that $f:[0,\infty)\rightarrow [0,\infty)$ is a continuous and increasing function such that $f(u)\leq A u^{\alpha}$, $A>0$, and $0 \leq u_0 \in L^{\gamma}(\Rn)$, for some $\gamma \in [1,\infty)$. Then the Cauchy problem \eqref{intro.prob}  has a global classical solution under some extra assumptions given explicitly in Theorem \ref{thmgl}.
    
    (iv) In particular, for $\alpha>\alpha_F=1+\frac{2}{q}$ and $f$ as in (iii), let $0 \leq u_0 \in L^{1}(\Rn)$ be sufficiently small. Then the Cauchy problem \eqref{intro.prob} has a global classical solution in $L^1(\Rn)$.
    \end{thm}

We note that in parts (i), (ii) in the above theorem, when  referring to a (global) distributional solution, we mean a function $u \in L^{p}_{\text{loc}}(((0,\infty);\Rn))$ which satisfies \eqref{inEQ:main/thm} in $\mathcal{D}'((0,\infty);\Rn)$. 

Parts (i) and (ii) above are the contents of Theorem \ref{thmnon}. Part (iii) is given in Theorem \ref{thmgl}, and Part (iv) in Remark \ref{REM:L1}.

Our proofs rely on heat kernel estimates for the fundamental solution of $\mathcal{H}=\partial_{t}-\L$ from \cite{BB23}, an idea that was  adopted also in the work \cite{Ruzhansky} on the study of the semi-linear heat equation  on a sub-Riemannian manifold $M$, when $\L$ is the sub-Laplacian in that setting.  When $M$ is a unimodular Lie group, then the critical Fujita exponent was obtained in \cite{Ruzhansky}  via the heat kernel estimates in \cite{VCSC92}. More recently, in \cite{CKR24} the authors extended the results of \cite{Ruzhansky} in the latter case by considering time-dependent non-linearities.  In \cite{FHS12}, the authors studied the heat equation with a forcing term and a condition on the volume growth of balls, assuming uniform polynomial growth. Under this condition, the heat kernel satisfies the following two-sided Gaussian-type estimate:
$$C_{1} t^{-\frac{\alpha}{\beta}} \exp\left(-\frac{\rho d^{2}(x,y)}{s}\right)\leq \Ga(0,x;s,y) \leq C_{2} t^{-\frac{\alpha}{\beta}} \exp\left(-\frac{d^{2}(x,y)}{\rho s}\right),$$
where $d(x,y)$ is the metric between points $x$ and $y$, $\rho>0, $ $\alpha,\beta>0$ and $C_{1}, C_{2}>0$. 

The main novelty in this paper is that we obtain the critical Fujita exponent $\alpha_F$ without assuming an underlying group structure on $\Rn$ which is matching the vector fields, and we do not have a lower estimate of the heat kernel as in the previous formula. However, provided that this structure exists, our case includes the case of a sub-Laplacian on a stratified group in the sense that under Assumption \ref{as1} on $X$, which is roughly speaking the H\"ormander rank condition at $0$ together with a homogeneity assumption, and if ${\rm rank}({\rm Lie}\{X\}(x))=n$, then $\L$ is the sub-Laplacian on a stratified group. The interesting case is when the latter equality is  not satisfied, and say we  have ${\rm rank}({\rm Lie}\{X\}(x))=N$, with $N>n,$ since then the  heat kernel estimates of $\mathcal{H}=\partial_t-\L$ are obtained via the ones corresponding to the ``lifted'' (in the sense of \cite[page 3]{BB23}, see also Definition \ref{def:lifting} below) version of $\mathcal{H}$, say $\tilde{\mathcal{H}}$, of the form $\tilde{\mathcal{H}}=\partial_t-\L_\G$, where $\L_\G$  is now a sub-Laplacian on a stratified group $\G\equiv \mathbb{R}^N$, see \cite[Theorem 1.2]{BB23} or Theorem \ref{thm1.2} below. The purpose of Assumption \ref{as1} is exactly that it allows for the fundamental solution of the ``lifted'' operator $\tilde{\mathcal{H}}$ to exist. 

Let us stress that this idea of ``lifting'' from Biagi and Bonfiglioli in \cite{BB23} should  not be confused with the celebrated Rothschild-Stein lifting theorem \cite{RS}, which is more a local tool. They both include the idea of lifting the pursued analysis to that of a manifold of higher dimension, though in the case of \cite{BB23} the lifting procedure is global since it is developed with the aim of studying the fundamental solution of $\tilde{\mathcal{L}}$ on the whole space, and the formal definition is as follows:

\begin{defn}[Lifted operator]\label{def:lifting}
    Let $\HH$ be a smooth linear partial differential operator on $\mathbb{R}^{n+1}$. We say that the partial differential operator $\tilde{\HH}$ defined on $\mathbb{R}^{n+1}\times \mathbb{R}^{p}$ is a lifting of $\HH$ if $\tilde{\HH}$ has smooth coefficients in $\mathbb{R}^{n+1}\times \mathbb{R}^{p}$, and for every smooth function $f \in C^{\infty}(\mathbb{R}^{n+1})$ the operators $\tilde{\HH}$ and $\HH$ coincide on the intersection of their domains; i.e., 
    \[
    \tilde{\HH}(f \circ \pi)(z,\xi)=(\HH f)(z)\,,\quad \hbox{for all $(z,\xi) \in \mathbb{R}^{n+1}\times \mathbb{R}^p$}\,,
    \]
    where $\pi(z,\xi)=z$ is the projection of $\mathbb{R}^{n+1}\times \mathbb{R}^p$ onto its subspace $\mathbb{R}^{n+1}$.
\end{defn}

The structure of this work is as follows: in Section \ref{PR}, we give some preliminary results on the  properties of the heat kernel of homogeneous H\"ormander operators which, if lifted in the sense of \cite{BB23}, can be viewed as the sub-Laplacian of a stratified group adding to this the derivative with respect to a real variable. 
  In Section \ref{MR}, we study the global well-posedness of the heat equation with initial data \eqref{intro.prob}, and the existence of a classical solution of the differential inequality \eqref{inEQ:main/thm}.

\section{Preliminaries}\label{PR}

In this section, we clarify the assumptions of our setting, i.e., homogeneity assumptions in $\Rn$, and assumptions on the differential operator $\mathcal{H}$. We also describe the properties of the heat kernel of $\mathcal{H}$ which are given via the ones of the lifted version $\tilde{\mathcal{H}}$ of $\mathcal{H}$.
\subsection{Assumptions  on the vector fields}
Note that the subsequent analysis does not require fixing a group law on the underlying manifold $\Rn$  matching the behavior of the vector fields. However, we do assume the existence of  a family of non-isotropic dilations $\{\delta_{\lambda}\}_{\lambda>0}$ on $\Rn$; that is  a map $\delta_{\lambda}:\Rn\rightarrow \Rn$ of the form
        \begin{equation}\label{def:dilation}
            \delta_{\lambda}(x)=(\lambda^{\sigma_{1}}x_{1},\ldots,\lambda^{\sigma_{n}}x_{n}),
        \end{equation}
 where $1=\sigma_{1}\leq \ldots\leq \sigma_{n}$ are integers.  
 
 If we denote by $T(\Rn)$ the Lie algebra of the vector fields on $\Rn$, then for $U \subset T(\Rn)$ we denote by $\rm Lie\{U\}$ the set:
\[
{\rm Lie}\{U\}=\bigcap \mathcal{U}\,, \quad \hbox{where  $\mathcal{U}$ is a sub-algebra (with respect to $[\cdot,\cdot]$) of $T(\Rn)$ with $U\subset \mathcal{U}$.}
\]
Then for $x\in \Rn$ we define
\[
{\rm rank}({\rm Lie}\{U\})(x):= {\rm dim}\{XI(x), X \in {\rm Lie}\{U\}\}\,,
\]
where $I$ denotes the identity map on $\Rn$.

Recall that the differential operator $X \in T(\Rn)$ is called homogeneous of degree $\ell \in \mathbb{R}$, if for every $\phi \in C^{\infty}(\Rn)$, and every $\lambda>0$, we have 
\[
X (\phi \circ \delta_\lambda)=\lambda^{\ell} (X\phi) \circ \delta_\lambda\,.
\]

\begin{as}\label{as1}
    Let $X = \{X_1,\dots, X_m\}$ be a set of smooth and linearly independent \footnote{Note that the vector fields $X,Y$ can be linearly independent as vector fields on $\Rn$ even if  $XI(x),YI(x)$ are linearly dependent on $\Rn$ for some $x \in \Rn$. A simple example of such vector fields is the set $\{X=\partial_{x_1}, Y=x_1\partial_{ x_2}\}$ for which we have ${\rm rank( Lie}\{X,Y\}(x))=2$ for all $x \in \mathbb{R}^2$.} vector fields on $\Rn$ satisfying the following assumptions:
    \begin{itemize}
        \item [(H1)] There exists a family of non-isotropic dilations of the form \eqref{def:dilation}, such that $X_1,...,X_m$ are homogeneous of degree 1 with respect to $\delta_\lambda$.
        \item[(H2)] The set $X$ satisfies H\"{o}rmander's rank condition at $0$, i.e.,
\begin{equation*}
{\rm rank}({\rm Lie}\{X\}(0))=n.
\end{equation*}
    \end{itemize}
\end{as}
\begin{rem}
    Note that for vector fields satisfying condition (H1), condition (H2) implies that H\"ormander's rank condition is satisfied at any $x \in \Rn$, and 
    \[
    {\rm rank}({\rm Lie}\{X\}(0))\leq {\rm rank}({\rm Lie}\{X\}(x))\,,\,\,\text{for all}\,\,\, x \in \Rn,
    \]
    see also \cite[Remark 1.1]{BB23}
\end{rem}
\begin{ex}
   For $\gamma\in \mathbb{N}$ the set  $X=\{X_{1}=\frac{\partial}{\partial x_{1}}, X_{2}=x_{1}^{\gamma}\frac{\partial}{\partial x_{2}}\}$ satisfies Assumption \ref{as1}. Indeed, $X_1,X_2$ are linearly independent, and for $\delta_{\lambda}(x):=(\lambda x_{1},\lambda^{\gamma+1}x_{2})$ they satisfy (H1), while also ${\rm rank}({\rm Lie}\{X\}(0))=2$. The operator $\mathcal{L}=X_{1}^{2}+X^{2}_{2}$ is a special case of a family of Baouendi-Grushin operators, and cannot be regarded as the sub-Laplacian of a stratified group since $X_1$ and $X_2$ are not left invariant with respect to any group law on $\mathbb{R}^2$ even when $\gamma=1$.
\end{ex}

\subsection{Lifting of the vector fields and estimates on the associated CC-balls}

Here we present the detailed properties of the vector fields that we consider in this work, which coincide with those in \cite{BB23}. 

\begin{thm}[Theorem 1.2, \cite{BB23}]\label{thm1.2} Assume that the system of vector fields $X$ satisfies Assumption \ref{as1} and $N = \dim(\text{Lie}\{X\})$. We have
  \begin{itemize}
      \item [(1)] If $N= n$, then there exists a stratified group $\G\equiv \Rn$ with  dilations $\delta_{\lambda}$ as in (H1) such that $X=\text{Lie}(\G)$; i.e., $X$ generates the Lie algebra of $\G$.  
 \item [(2)] If $N > n$, then there exist a stratified group $\G \equiv \mathbb{R}^N$  and a system of vector fields $Z=\{Z_1, \ldots, Z_m\}$ such that $Z=\text{Lie}(\G)$; that is $Z$ generates the Lie algebra of $\G$ and $Z_i$ is a lifting, in the sense of Definition \ref{def:lifting}, of $X_i$, $i = 1,\ldots,m$.  
  \end{itemize}
\end{thm}
\begin{rem}
    Theorem \ref{thm1.2} implies in particular that if  $dim(\text{Lie}\{X\})=N= n$, then the operator $\L$  is exactly the sub-Laplacian on $\G\equiv \Rn$, while if $dim(\text{Lie}\{X\})=N > n$, then $\L$ can be lifted to a sub-Laplacian $\L_{\G}:=\sum\limits_{j=1}^{m}Z_j^2$ of $\G\equiv \mathbb{R}^N$.
\end{rem}

We recall, see \cite[Section 2]{BB23}, that given a system $X$ on $\Rn$ satisfying Assumption \ref{as1} with ${\rm rank}({\rm Lie}\{X\}(x))=N>n$ there exists a stratified group $\G =(\mathbb{R}^{N} , *, \delta_{\lambda})$ and $\mathbb{R}^N$ can be decomposed as  $\mathbb{R}^{N} = \Rn_{x} \times \mathbb{R}^{p}_{\xi}$ so that the first $n$ variables denoted by $x$ are the ``old'' ones and the extra $p$ variables denoted by $\xi$ are the ``new'' ones. Moreover
\begin{enumerate}[label=\roman*)]
      \item $\G$ is a homogeneous group with $m$ generators: the $m$ vector fields $Z_1,\ldots,Z_m$ that are lifting the $m$ original ones in $X$ generate it; i.e. ${\rm Lie}(\G)$ is spanned (as a vector space) by all iterated brackets of elements of $X$;
    \item $\G$ is a nilpotent group of step $\sigma_n$, where $\sigma_n$ is the one appearing in \eqref{def:dilation};
    \item the induced dilations on $\G$ are of the form 
    \[
     D_{\lambda}(x,\xi)=(\lambda^{\sigma_{1}}x_{1},\ldots,\lambda^{\sigma_{n}}x_{n},\lambda^{\sigma^{*}_{1}}\xi_{1},\ldots,\lambda^{\sigma^{*}_{p}}\xi_{p}),\,\,\,(x,\xi)\in \RN=\Rn_x\times\mathbb{R}^{p}_{\xi};
    \]
    \item we have three homogeneous dimensions that naturally arise
    \begin{equation}\label{homdim}
q:=\sum_{j=1}^{n}\sigma_{j},\,\,q^{*}:=\sum_{j=1}^{p}\sigma^{*}_{j},\,\,\,Q=q+q^{*},
\end{equation}
corresponding to $(\Rn,\delta_{\lambda})$, $(\Rp,\delta^{*}_{\lambda})$ and $(\RN,D_{\lambda})$, respectively.
\end{enumerate}

The system $X$ gives rise to the notion of the Carnot-Carath\'{e}odory distance, or in short, CC-distance.
\begin{defn}
    Assume that $X=\{X_{1},\ldots,X_{m}\}$ is a family of smooth vector fields on $\Rn$ satisfying H\"{o}rmander's rank condition at every point of $\Rn$. Then for any two points $x, y \in \Rn$ we can define the CC-distance $d_{X} (x, y)$ associated with $X$ as
    \begin{equation*}
        d_{X}(x,y)=\inf\{l(\gamma): \gamma \in \mathcal{S}(X)\,, \gamma(0)=x,\gamma(1)=y\},
    \end{equation*}
    where $\mathcal{S}(X)$ is the set of the absolutely continuous curves $\gamma : [0, 1] \rightarrow \Rn$ satisfying
    \begin{equation*}
        \gamma'(t)=\sum_{j=1}^{m}a_{j}(t)X_{j}(\gamma(t)),\,\,\,\text{with} \,\,|a_{j}(t)|\leq l(\gamma)\,\, \text{for all} \,\,j=1,\ldots,m\,,
    \end{equation*}
    where $l(\gamma)$ is the length of $\gamma$. 
\end{defn}
We denote by $B_X(x,\rho)$  the ball of radius $\rho$ centered at $x\in \Rn$ with respect to the CC-distance, i.e. 
\begin{equation*}
    B_{X}(x,\rho):=\{y\in\Rn:d_{X}(x,y)< \rho\}.
\end{equation*}
Thanks to assumption (H1), $d_{X}$ and $B_{X}$ possess homogeneity properties:
\begin{enumerate}\label{propcc}
    \item $d_{X}(\delta_{\lambda}(x),\delta_{\lambda}(y))=\lambda d_{X}(x,y);$
    \item $y\in B_{X}(x,r) \Longleftrightarrow \delta_{\lambda}(y)\in B_{X}(\delta_{\lambda}(x),\lambda r)$;
    \item $|B_{X}(\delta_{\lambda}(x),\lambda r)|=\lambda^{q} |B_{X}(x, r)|.$
\end{enumerate}
From \cite[Theorem B]{BBB}, we have  two-sided estimates for the volume of $B_{X}(x,\rho)$ :
\begin{enumerate}[label=\roman*)]
    \item  For $x \in \Rn$, $\rho>r>0$, and for $q$ as in \eqref{homdim}, we have
    \begin{equation}\label{prop.ball}
        \gamma_{1}\left(\frac{\rho}{r}\right)^{n}\leq\frac{|B_{X}(x,\rho)|}{|B_{X}(x,r)|}\leq\gamma_{2}\left(\frac{\rho}{r}\right)^{q}\,,
  \end{equation}
  
   for some $\gamma_1,\gamma_2>0$, see also [\cite{BB21}, Remark 3.9].
    \item For $x \in \Rn$, $r>0$, and for $q$ as in \eqref{homdim}, we have
    \begin{equation}\label{ocencc1}
      Cr^{q} \leq |B_{X}(x,r)|\,,
   \end{equation}
   for some  $C>0$.
   
\end{enumerate}
\subsection{Properties of the heat kernel of \texorpdfstring{$\mathcal{H}$}{H}}\label{subsec:heatk}

Recall that by $\mathcal{H}=\partial_{t}-\L$ with $\L=\sum\limits_{j=1}^{m}X_{j}^{2}$, where $X_j$ are smooth vector fields in $X$ which satisfy Assumption \ref{as1}, we denote the heat operator on $\mathbb{R}^{n+1}=\Rn{_x} \times \mathbb{R}_t$.  Biagi-Bonfiglioli in \cite{BB23} and Biagi-Bonfiglioli-Bramanti in \cite{BBB} studied, among other things, the properties of the heat kernel $\Ga$ of $\mathcal{H}$ in terms of the one of the lifted (in the $n$ variables) operator $\partial_t-\L_{\G}$, where $\L_\G=\sum\limits_{j=1}^{m}Z_j$ is the lifted operator as in Theorem \ref{thm1.2}. Among these properties, those that are useful to our analysis here are summarised below:

For $x,y\in \Rn$ and for every $s>0$, there exists $\rho>1$ such that
\begin{equation}\label{est2}
           \frac{1}{\rho|B_{X}(x,\sqrt{s})|} \exp\left(-\frac{\rho d^{2}_{X}(x,y)}{s}\right)\leq\Ga(0,x;s,y)
       \leq \frac{\rho}{|B_{X}(x,\sqrt{s})|} \exp\left(-\frac{ d^{2}_{X}(x,y)}{\rho s}\right)\,,
    \end{equation}
see \cite[Theorem 2.4, (i)]{BBB}. 

Moreover, by \cite[Theorem 1.4]{BB23}, we have
\begin{enumerate}[label=\roman*)]
    \item $\Ga\geq0$ and we have
\begin{equation}\label{EQ:s<t}
   \Ga(t,x;s,y)=0,\,\,\,\,s\leq t; 
\end{equation}
\item $\Ga$ is symmetric with respect to the space variables, i.e.
\begin{equation}\label{sym2}
\Ga(t,x;s,y) = \Ga(t,y;s,x);    
\end{equation}
\item for  fixed $(t,x) \in \mathbb{R}^{n+1},$ we have 
\begin{equation}\label{d}
    \int_{\Rn}\Ga(t,x;s,y)dy = 1, \,\,\,\text{whenever}\,\,\, s > t;
\end{equation}
\item for $x, y \in \Rn$ and $s, t > 0$, we have the reproduction formula
\begin{equation}\label{f}
\Ga(0,y;t+s,x)=\int_{\Rn}\Ga(0,w;t,x)\Ga(0,y;s,w)dw;
\end{equation}
\end{enumerate}
Finally, see \cite[Theorem 1.4, (ix)]{BB23}, if $\varphi\in C(\Rn)$ is bounded, then 
\begin{equation}
    u(t,x):=\int_{\Rn}\Ga(0,y;t,x)\varphi(y)dy,\,\,\, (t,x) \in \mathbb{R}_{+}\times\Rn,
\end{equation}
is the (unique) bounded classical solution of the homogeneous Cauchy problem 
\begin{equation*}
    \begin{cases}
        \mathcal{H}u(t,x)=\partial_t u(t,x)-\L u(t,x)=0,\,\,\,(t,x) \in (0,+\infty)\times \mathbb{R}^{n},\\
        u(0,x)=\varphi(x),\,\,\,x\in\Rn.
    \end{cases}
\end{equation*}
It follows that the function satisfying 
\[
         v(t,x)= \int_{\Rn}\Gamma(0,y;t,x)v_{0}(y)dy+\int_{0}^{t}\int_{\Rn}\Gamma(0,y;t-\tau,x)f(v(\tau,y))dyd\tau,
\]
is the unique solution of the differential equality 
\begin{equation*}
    \begin{cases}
       \mathcal{H}v=v_t-\L v = f(v),\,\,\,(t,x) \in (0,+\infty)\times \mathbb{R}^{n},\\
        v(0,x)=v_0(x),\,\,\,x\in\Rn.
    \end{cases}
\end{equation*}

\section{Constant coefficients case}\label{MR}
In this section, we show the existence and non-existence results for the Cauchy problem \eqref{intro.prob} where the differential operator arises as the sum of squares of vector fields satisfying Assumption \ref{as1}. 

Firstly, we will show $L^{p}(\mathbb{R}^{n})$--$L^{q}(\mathbb{R}^{n})$ decay of the heat semigroup.
\begin{thm}\label{THM:LpnormGamma}
    Assume that $1\leq p\leq r\leq +\infty$ and $u_{0}\in L^{p}(\mathbb{R}^{n})$. Then, we have
    \begin{equation}\label{LpLq}
        \left\|\int_{\Rn}\Gamma(0,y;t,\cdot)u_{0}(y)dy\right\|_{L^{r}(\Rn)}\leq Ct^{-\frac{q}{2}\left(\frac{1}{p}-\frac{1}{r}\right)}\|u_{0}\|_{L^{p}(\Rn)}\,,
    \end{equation}
    for all $t>0$.
\end{thm}
\begin{proof}
    Using the upper estimate of the heat kernel and \eqref{ocencc1}, we have
    $$\left\|\int_{\Rn}\Gamma(0,y;t,\cdot)u_{0}(y)dy\right\|_{L^{\infty}(\Rn)}\leq Ct^{-\frac{q}{2}}\|u_{0}\|_{L^{1}(\Rn)},$$
and using the integral Minkowski inequality and \eqref{d}, we get 
$$\left\|\int_{\Rn}\Gamma(0,y;t,\cdot)u_{0}(y)dy\right\|_{L^{p}(\Rn)}\leq \|u_{0}\|_{L^{p}(\Rn)}.$$

Finally, using the Riesz-Thorin interpolation theorem, we obtain \eqref{LpLq}.
\end{proof}
The existence result for the solution to the Cauchy problem \eqref{intro.prob} essentially  relies on the assumption on the  initial data $u_0$ demonstrated in the next theorem. 
\begin{thm}\label{thmgl}
    Let $X=\{X_{1},\ldots,X_{m}\}$ be a system of vector fields on $\Rn$ that satisfy Assumption \ref{as1}. Suppose that $f$ is a nonnegative, continuous, and increasing function such that $f(u)\leq Au^{\alpha}$ with $A>0$ and $\alpha\in(1,\infty)$, for all $u \geq 0$. If $0\leq u_{0}\in L^{\gamma}(\Rn)$, for some $\gamma\in[1,\infty)$, and 
    \begin{equation}\label{congl}
\int_{0}^{\infty}\left\|\int_{\Rn}\Gamma(0,y;t,\cdot)u_{0}(y)dy\right\|^{\alpha-1}_{L^{\infty}(\Rn)}dt<\frac{1}{A(\alpha-1)},
    \end{equation}
    then there exists a non-negative curve $u:[0,\infty)\rightarrow L^{\gamma}(\Rn)$ which is a global solution of \eqref{intro.prob}. 
    Additionally, there exists $C>1$ such that
    
    \begin{equation}\label{thm.3.1.conc}
        \int_{\Rn}\Gamma(0,y;t,x)u_{0}(y)dy\leq u(t,x)\leq C\int_{\Rn}\Gamma(0,y;t,x)u_{0}(y)dy,
    \end{equation}
    for a.e. $(t,x)\in \mathbb{R}_{+}\times \Rn$.
\end{thm}

\begin{rem}\label{REM:L1}
   The decay estimate \eqref{LpLq} allows us to verify the global existence condition \eqref{congl}; indeed, for $u_0 \in L^1(\mathbb{R}^n) \cap L^\infty(\mathbb{R}^n)$, if we write
\[
\int_0^\infty \left\|\int_{\Rn}\Gamma(0,y;t,\cdot)u_{0}(y)dy \right\|_{L^\infty({\Rn})}^{\alpha-1} dt = \int_0^1 + \int_1^\infty\,,
\]
by choosing $p=\infty$ if $t \in [0,1]$ and $p=1$ if $t>1$, we obtain
\[
\int_0^\infty \left\|\int_{\Rn}\Gamma(0,y;t,\cdot)u_{0}(y)dy \right\|_{L^\infty({\Rn})}^{\alpha-1} dt \leq \|u_0\|_{L^\infty({\Rn})}^{\alpha-1} + C \|u_0\|_{L^1({\Rn})}^{\alpha-1} \int_1^\infty t^{-\frac{q}{2}(\alpha-1)} dt.
\]
Since for $\alpha > 1 + \frac{2}{q}$ the second integral converges, for sufficiently small $\|u_0\|_{L^1(\Rn)}$ and $\|u_0\|_{L^\infty(\Rn)}$, condition \eqref{congl} is satisfied.
\end{rem}

Also, in the following theorem we show if initial data is sufficiently small and $\alpha>1+\frac{2}{q}$, we can also obtain the global existence result, and if the data is bounded by a Gaussian, the solution also stays bounded by the Gaussian.
\begin{thm}\label{corglobex}
    Let $X=\{X_{1},\ldots,X_{m}\}$ be a system of vector fields on $\Rn$ that satisfies Assumption \ref{as1}. For $f$ and $u_0$ satisfying the assumptions of Theorem \ref{thmgl}, and if additionally $0<u_{0}(x)\leq \theta\Gamma(0,0;\varrho,x)$ for some $\theta>0$ sufficiently small, $\varrho>0$, and $\alpha>1+\frac{2}{q}$, where $q$ is defined in \eqref{homdim}, then there exists a non-negative curve $u:[0,\infty)\rightarrow L^{\gamma}(\Rn)$ for all $\gamma\in[1,\infty]$, which is a global solution of \eqref{intro.prob}. In addition,
    \begin{equation*}
         u(t,x)\leq C\theta\Gamma(0,0;t+\varrho,x),
    \end{equation*}
    where $C>0$ depends only on $X$. 
\end{thm}
\begin{proof}[Proof of Theorem \ref{corglobex}]
By the assumption on $u_0$ we have
 \begin{equation*}
        \begin{split}
\int_{0}^{\infty}&\left\|\int_{\Rn}\Gamma(0,y;t,\cdot)u_{0}(y)dy\right\|^{\alpha-1}_{L^{\infty}(\Rn)}dt\leq \theta^{\alpha-1}\int_{0}^{\infty}\left\|\int_{\Rn}\Gamma(0,y;t,\cdot)\Gamma(0,0;\varrho,y)dy\right\|^{\alpha-1}_{L^{\infty}(\Rn)}dt\\&
\stackrel{\eqref{f}}=\theta^{\alpha-1}\int_{0}^{\infty}\left\|\Gamma(0,0;t+\varrho,\cdot)\right\|^{\alpha-1}_{L^{\infty}(\Rn)}dt\\&
\stackrel{\eqref{est2}}\leq C\theta^{\alpha-1}\int_{0}^{\infty}\left\|\frac{\rho}{|B_{X}(\cdot,\sqrt{t+\varrho})|} \exp\left(-\frac{ d_{X}(\cdot,0)}{\rho (t+\varrho)}\right)\right\|^{\alpha-1}_{L^{\infty}(\Rn)}dt\\&
\stackrel{\eqref{ocencc1}}\leq C\theta^{\alpha-1}\rho^{\alpha-1}\int_{\varrho}^{\infty}z^{-\frac{q}{2}(\alpha-1)}dz\\&
\stackrel{\alpha>1+\frac{2}{q}}\leq \frac{1}{A(\alpha-1)},
    \end{split}
    \end{equation*}
   for some $A>0$, since $\theta$ is sufficiently small. By Theorem \ref{thmgl} there exists a global solution to \eqref{intro.prob} and we have 
     \begin{equation*}
    \begin{split}
        u(t,x)&\stackrel{\eqref{thm.3.1.conc}}\leq C\int_{\Rn}\Gamma(0,y;t,x)u_{0}(y)dy\\&
        \leq C\theta\int_{\Rn}\Gamma(0,y;t,x)\Gamma(0,0;\varrho,y)dy\\&
        \stackrel{\eqref{f}}=C\theta\Gamma(0,0;t+\varrho,x)\,,
        \end{split}
\end{equation*}
completing the proof. 
\end{proof}
The proof of Theorem \ref{thmgl} is based on a monotone sequence argument that was introduced by Weissler in \cite{Weissler}.
\begin{proof}[Proof of Theorem \ref{thmgl}]
  Note that by virtue of Fubini's theorem we interchange the order of integration in several parts of  the present proof. The symmetry of the kernel $\Ga$ with respect to the space variables, see \eqref{sym2}, has also been used here without being mentioned. 
  
By \eqref{congl} the quantity  
\begin{equation*}
    \chi(t):=\left(1-A(\alpha-1)\int_{0}^{t}\left\|\int_{\Rn}\Gamma(0,y;\tau,\cdot)u_{0}(y)dy\right\|^{\alpha-1}_{L^{\infty}(\Rn)}d\tau\right)^{-\frac{1}{\alpha-1}}\,,
\end{equation*}
is well defined.  We have $\chi(0)=1$ and $\chi'(t)=A\left\|\int_{\Rn}\Gamma(0,y;t,\cdot)u_{0}(y)dy\right\|^{\alpha-1}_{L^{\infty}(\Rn)}\chi^{\alpha}(t)$. Solving this ODE with initial conditions gives 
\begin{equation}\label{x^a}
    \chi(t)=1+A\int_{0}^{t}\left\|\int_{\Rn}\Gamma(0,y;\tau,\cdot)u_{0}(y)dy\right\|^{\alpha-1}_{L^{\infty}(\Rn)}\chi^{\alpha}(\tau)d\tau.
\end{equation}
Let $h:[0,\infty) \rightarrow L^\gamma(\Rn)$ be a continuous curve for which we have 
\begin{equation}\label{eq:ass.h}
 \int_{\Rn}\Gamma(0,y;t,x)u_0(y)dy  \leq h(t,x)\leq \chi(t)\int_{\Rn}\Gamma(0,y;t,x)u_0(y)dy\,, \qquad \mbox{for all } t>0\,.
\end{equation}
We define the operator $\mathcal{K}$ acting on a function, say $v=v(t)$ where $t \in [0,\infty)$, as follows
    \begin{equation*}
\mathcal{K}v(t):=\int_{\Rn}\Gamma(0,y;t,x)u_0(y)dy+\int_{0}^{t}\int_{\Rn}\Gamma(0,y;t-\tau,x)f(v(\tau,y))dyd\tau\,,
    \end{equation*}
    where $u_0$ is the initial data from \eqref{intro.prob}. Using that $f(u)\leq Au^{\alpha}$  we get 
    \begin{equation*}
        \begin{split}
&\mathcal{K}h(t,x)=\int_{\Rn}\Gamma(0,y;t,x)u_0(y)dy+\int_{0}^{t}\int_{\Rn}\Gamma(0,y;t-\tau,x)f(h(\tau,y))dyd\tau \\&
\leq \int_{\Rn}\Gamma(0,y;t,x)u_0(y)dy+A\int_{0}^{t}\int_{\Rn}\Gamma(0,y;t-\tau,x)h^{\alpha}(\tau,y)dyd\tau\\&
\stackrel{\eqref{eq:ass.h}}\leq \int_{\Rn}\Gamma(0,y;t,x)u_0(y)dy
+A\int_{0}^{t}\int_{\Rn}\Gamma(0,y;t-\tau,x)\chi^{\alpha}(\tau)\left[\int_{\Rn}\Gamma(0,z;\tau,y)u_{0}(z)dz\right]^{\alpha}dyd\tau\\&
\leq \int_{\Rn}\Gamma(0,y;t,x)u_{0}(y)dy+A\int_{0}^{t}\left\|\int_{\Rn}\Gamma(0,z;\tau,x)u_{0}(z)dz\right\|_{L^{\infty}(\Rn)}^{\alpha-1}\chi^{\alpha}(\tau)\\&
\times\int_{\Rn}\Gamma(0,y;t-\tau,x)\int_{\Rn}\Gamma(0,z;\tau,y)u_{0}(z)dzdyd\tau
\end{split}
    \end{equation*}
    \begin{equation}\label{globocen}
        \begin{split}
&=\int_{\Rn}\Gamma(0,y;t,x)u_{0}(y)dy+A\int_{0}^{t}\left\|\int_{\Rn}\Gamma(0,z;\tau,\cdot)u_{0}(z)dz\right\|_{L^{\infty}(\Rn)}^{\alpha-1}\chi^{\alpha}(\tau)\\&
\times\int_{\Rn}u_{0}(z)\int_{\Rn}\Gamma(0,y;t-\tau,\cdot)\Gamma(0,z;\tau,y)dydzd\tau\\&
\stackrel{\eqref{f}}=\int_{\Rn}\Gamma(0,y;t,\cdot)u_{0}(y)dy+A\int_{0}^{t}\left\|\int_{\Rn}\Gamma(0,z;\tau,\cdot)u_{0}(z)dz\right\|_{L^{\infty}(\Rn)}^{\alpha-1}\chi^{\alpha}(\tau)\\&
\times\int_{\Rn}\Gamma(0,z;t,\cdot)u_{0}(z)dzd\tau\\&
=\left(1+A\int_{0}^{t}\left\|\int_{\Rn}\Gamma(0,y;\tau,x)u_{0}(y)dy\right\|_{L^{\infty}(\Rn)}^{\alpha-1}\chi^{\alpha}(\tau)d\tau\right)\int_{\Rn}\Gamma(0,z;t,x)u_{0}(z)dz\\&
=\chi(t)\int_{\Rn}\Gamma(0,z;t,x)u_{0}(z)dz,
        \end{split}
    \end{equation}
    where for the last equality we have used \eqref{x^a}. 
    We consider the sequence  of functions $\{h_{k}\}_{k=1}^{\infty}$ in $t \in [0,\infty)$ that we define inductively as follows: 
    \begin{equation*}
        \begin{cases}
            h_{0}(t,x)=\int_{\Rn}\Gamma(0,y;t,x)u_{0}(y)dy,\\
            h_{k+1}(t,x)=\mathcal{K}h_{k}(t,x),\,\,\,\,k\geq1\,.
        \end{cases}
    \end{equation*}
   We  have 
    \begin{equation*}
       \int_{\Rn}\Gamma(0,y;t,x)u_{0}(y)dy\leq h_{0}(t,x)\leq \chi(t)\int_{\Rn}\Gamma(0,y;t,x)u_{0}(y)dy, 
    \end{equation*}
   for a.e. $t \in [0,\infty)$. Additionally, using induction arguing as in \eqref{globocen} we obtain
     \begin{equation}\label{EQ:h_k}
       \int_{\Rn}\Gamma(0,y;t,x)u_{0}(y)dy\leq h_{k}(t,x)\leq \chi(t)\int_{\Rn}\Gamma(0,y;t,x)u_{0}(y)dy,\,\,\,\, 
    \end{equation}
for all $t \in [0,\infty)$ for each $k$. It is straightforward that $h_{0}\leq h_{1}$. Also, by the monotonicity of $f$, we have that  if $h_k \leq h_{k+1}$ then $\mathcal{K}h_k \leq \mathcal{K}h_{k+1}$ for all $k\geq 0$. Thus, by induction, one obtains that $h_k \leq h_{k+1}$, and so the dominated convergence theorem yields that $h_k$ converges in $L^\gamma$ to some function, say $u=u(t)$. Hence, by \eqref{EQ:h_k} we obtain 
\[
 \int_{\Rn}\Gamma(0,y;t,x)u_{0}(y)dy\leq u(t,x)\leq \chi(t)\int_{\Rn}\Gamma(0,y;t,x)u_{0}(y)dy\,,\qquad \mbox{for all } t>0\,.
\]
We claim that $u(t)$ is the global solution of the Cauchy problem \eqref{intro.prob}. To this end, let us fix $t \in [0,\infty)$, and let $\tau \in [0,t)$. Let us consider the functions $F_{k,t}$ defined by $\tau \mapsto \int_{\Rn}\Gamma(0,y;t-\tau,\cdot)f(h_{k}(\tau))dy$ for $k \in \mathbb{N}$. We estimate with $F_{k,t}(\tau)=F_{k,t}(\tau,\cdot)$ as a function of $x$,

\begin{equation*}
\begin{split}
   &F_{k,t}(\tau,x)= \int_{\Rn}\Gamma(0,y;t-\tau,x)f(h_{k}(\tau,y)dy\stackrel{f(u)\leq Au^{\alpha}}\leq  A\int_{\Rn}\Gamma(0,y;t-\tau,x)h^{\alpha}_{k}(\tau,y)dy\\&
     \stackrel{\eqref{EQ:h_k}}\leq A\int_{\Rn}\Gamma(0,y;t-\tau,x)\chi^{\alpha}(\tau)\left[\int_{\Rn}\Gamma(0,z;\tau,y)u_{0}(z)dz\right]^{\alpha}dy\\&
     \leq A\left\|\int_{\Rn}\Gamma(0,z;\tau,\cdot)u_{0}(z)dz\right\|_{L^{\infty}(\Rn)}^{\alpha-1}\chi^{\alpha}(\tau)
     \int_{\Rn}\Gamma(0,y;t-\tau,x)\int_{\Rn}\Gamma(0,z;\tau,y)u_{0}(z)dzdy\\&
\stackrel{\eqref{f}}=A\left\|\int_{\Rn}\Gamma(0,z;\tau,x)u_{0}(z)dz\right\|_{L^{\infty}(\Rn)}^{\alpha-1}\chi^{\alpha}(\tau)\int_{\Rn}\Gamma(0,z;t,x)u_{0}(z)dz=G_t(\tau,x)\,.
\end{split}
\end{equation*}
Hence the functions $F_{k,t}$ are dominated in $L^{1}(0,t;L^{\gamma}(\Rn))$. Now using the fact that $f$ is continuous and that $G_t(\tau)$ is in $L^{\gamma}(\Rn)$ for each $\tau \in (0,t)$, by the dominated convergence  theorem we get that for every $\tau$ the sequence $F_{k,t}(\tau)$ converges  to $$\int_{\Rn}\Gamma(0,y;t-\tau,x)f(u(\tau,y))dy.$$  That is, by the dominated convergence theorem for $L^{\gamma}(\Rn)$ function we deduce that 
\[
\lim_{k \rightarrow \infty} \int_{0}^{t}\int_{\Rn}\Gamma(0,y;t-\tau,x)f(h_{k}(\tau,y))dyd\tau=\int_{0}^{t}\int_{\Rn}\Gamma(0,y;t-\tau,x)f(u(\tau,y))dyd\tau,
\]
implying also that 
\[
u(t,x)=\lim_{k \rightarrow \infty} h_{k}(t,x)=\lim_{k \rightarrow \infty} \mathcal{K}h_k(t,x)=\mathcal{K}u(t,x)\,,
\]
that is $u$ is the global solution of \eqref{intro.prob}. The continuity of the operator $u:[0,\infty) \rightarrow L^{\gamma}(\Rn)$  follows by standard arguments, completing the proof.
\end{proof}
The non-existence result as in parts (i) and (ii) of our main theorem (Theorem \ref{THM.main}) is implied by the following. 
\begin{thm}\label{thmnon}
   Let $X=\{X_{1},\ldots,X_{m}\}$ be a system of vector fields on $\Rn$ that  satisfy Assumption \ref{as1}. Suppose that $f:[0,\infty)\rightarrow \mathbb{R}$ is a locally integrable function such that $f(u)\geq Bu^{\alpha}$ for some $B>0$. Let $0< u_{0}$ be a measurable function on $\Rn$. 
   
   (i) If $1<\alpha< 1+\frac{2}{q}$ where $q$ is defined in \eqref{homdim}, then  there is no nonnegative measurable global solution $u:\mathbb{R}_{+}\times \Rn\rightarrow [0,\infty]$ to the mild solution inequality
    \begin{equation}\label{non1}
u(t,x)\geq \int_{\Rn}\Gamma(0,y;t,x)u_{0}(y)dy+\int_{0}^{t}\int_{\Rn}\Gamma(0,y;t-\tau,x)f(u(\tau,y))dyd\tau.
    \end{equation}
   (ii) If $\alpha=1+\frac{2}{q}$
   then  there is no nonnegative measurable global solution $u:\mathbb{R}_{+}\times \Rn\rightarrow [0,\infty]$ to the mild solution equation
    \begin{equation}\label{non2}
u(t,x)= \int_{\Rn}\Gamma(0,y;t,x)u_{0}(y)dy+\int_{0}^{t}\int_{\Rn}\Gamma(0,y;t-\tau,x)f(u(\tau,y))dyd\tau.
    \end{equation}
\end{thm}

To prove this theorem we need the following lemma.
\begin{lem}\label{lem1}
     Let $X=\{X_{1},\ldots,X_{m}\}$ be a system of vector fields on $\Rn$ that  satisfy Assumption \ref{as1}. Suppose that $f:[0,\infty)\rightarrow \mathbb{R}$  is a locally integrable function such that $f(u)\geq Bu^{\alpha}$ for some $B>0$ and $\alpha\in(1,\infty)$. Let $v(t,x)\geq0$ for every $(t,x)\in[0,T]\times \Rn$ be measurable and satisfy 
     \begin{equation}\label{lemin}
         v(t,x)\geq \int_{\Rn}\Gamma(0,y;t,x)v_{0}(y)dy+\int_{0}^{t}\int_{\Rn}\Gamma(0,y;t-\tau,x)f(v(\tau,y))dyd\tau,
     \end{equation}
   for  a.e.  $(t,x)\in[0,T]\times\Rn$. Suppose that $v(t,x)<\infty$ for a.e.  $(t,x)\in[0,T]\times\Rn$. Then, we have
    \begin{equation}\label{lemcon}
        t^{\frac{1}{\alpha-1}}\left\|\int_{\Rn}\Gamma(0,y;t,\cdot)v_{0}(y)dy\right\|_{L^{\infty}(\Rn)}\leq C_{\alpha}:=(B(\alpha-1))^{-\frac{1}{\alpha-1}},
    \end{equation}
     for a.e. $t \in [0,T]$, where $v_0(y)=v(0,y)$.
\end{lem}
Lemma \ref{lem1} implies the following.
\begin{cor}\label{cor2}
    Let $X=\{X_{1},\ldots,X_{m}\}$ be a system of vector fields on $\Rn$ that  satisfy Assumption \ref{as1}. Suppose that $f:[0,\infty)\rightarrow \mathbb{R}$ is a locally integrable function such that $f(u)\geq Bu^{\alpha}$ for some $B>0$ and $\alpha\in(1,\infty)$. Let  $v(t,x)\geq0$ for every $(t,x)\in[0,T]\times \Rn$ be measurable and satisfy 
     \begin{equation}\label{cor3.5:as}
         v(t,x)= \int_{\Rn}\Gamma(0,y;t,x)v_{0}(y)dy+\int_{0}^{t}\int_{\Rn}\Gamma(0,y;t-\tau,x)f(v(\tau,y))dyd\tau,
     \end{equation}
    for a.e.  $(t,x)\in[0,T]\times\Rn$.   Then, we have
    \begin{equation}\label{cor.3.5.est}
        t^{\frac{1}{\alpha-1}}\left\|\int_{\Rn}\Gamma(0,y;t,\cdot)v(\tau, y)dy\right\|_{L^{\infty}(\Rn)}\leq C_{\alpha}:=(B(\alpha-1))^{-\frac{1}{\alpha-1}},
    \end{equation}
 for a.e.   $t\in[0,T-\tau]$ and $\tau\in[0,T]$.
\end{cor}
\begin{proof}[Proof of Corollary \ref{cor2}]
    Throughout this proof we use the property \eqref{sym2} of the heat kernel $\Ga$. For $\tau\in[0,T]$ we define  $\tilde{v}(t,x)=v(t+\tau,x)$. For a.e.  $(t,x)\in[0,T-\tau]\times \Rn$ we have 
    \begin{equation*}
        \begin{split}
       &\tilde{v}(t,x)=v(t+\tau,x)\\&
\stackrel{\eqref{cor3.5:as}}=\int_{\Rn}\Gamma(0,y;t+\tau,x)v_{0}(y)dy+\int_{0}^{t+\tau}\int_{\Rn}\Gamma(0,y;t+\tau-s,x)f(v(s,y))dyds\\&
\stackrel{\eqref{f}}=\int_{\Rn}\Gamma(0,x;t,w)\left[\int_{\Rn}\Gamma(0,w;\tau,y)v_{0}(y)dy\right]dw\\&
+\int_{0}^{\tau}\int_{\Rn}\Ga(0,x;t,w)\left[\int_{\Rn}\Gamma(0,w;\tau-s,y)f(v(s,y))dy\right]dwds\\&
+\int_{\tau}^{t+\tau}\int_{\Rn}\Gamma(0,y;t+\tau-s,x)f(v(s,y))dyds\\&
=\int_{\Rn}\Gamma(0,x;t,w)\left[\int_{\Rn}\Gamma(0,w;\tau,y)v_{0}(y)dy\right]dw\\&
+\int_{\Rn}\Ga(0,x;t,w)\left[\int_{0}^{\tau}\int_{\Rn}\Gamma(0,w;\tau-s,y)f(v(s,y))dyds\right]dw\\&
+\int_{\tau}^{t+\tau}\int_{\Rn}\Gamma(0,y;t+\tau-s,x)f(v(s,y))dyds
\end{split}
    \end{equation*}
\begin{equation*}
        \begin{split}
&\stackrel{\eqref{cor3.5:as}}=\int_{\Rn}\Ga(0,x;t,w)v(\tau,w)dw+\int_{\tau}^{t+\tau}\int_{\Rn}\Gamma(0,y;t+\tau-s,x)f(v(s,y))dyds\\&
=\int_{\Rn}\Ga(0,x;t,w)v(\tau,w)dw+\int_{0}^{t}\int_{\Rn}\Gamma(0,y;t-s,x)f(\tilde{v}(s,y))dyds,
        \end{split}
    \end{equation*}
where, for the last inequality, we have applied a change of variables and Fubini's theorem in several steps. Therefore, Lemma \ref{lem1} with $v_{0}(y)$ replaced by $v(\tau,y)$ and $T$ replaced by $T-\tau$ for a.e. $ \tau \in (0,T)$
completes the proof.
\end{proof}
\begin{proof}[Proof of Theorem \ref{thmnon}]
Below we use Fubini's theorem and the symmetric property \eqref{sym2} of $\Ga$ without explicitly mentioning. 

    First we prove the  non-existence result in the case $1<\alpha<1+\frac{2}{q}$ by contradiction.  Assume that there exists $u$ satisfying \eqref{non1}. Then by 
    Lemma \ref{lem1}, we have 
    \begin{equation}\label{Ca}
        t^{\frac{1}{\alpha-1}}\left|\int_{\Rn}\Gamma(0,y;t,x)u_{0}(y)dy\right|\leq C_{\alpha}\,, \quad \hbox{for all $t \in [0,T].$}
    \end{equation}
    For $\alpha,q$ as in the hypothesis we have
    \begin{equation}\label{asd1}
         \lim_{t \rightarrow \infty} t^{\frac{q}{2}}\left|\int_{\Rn}\Gamma(0,y;t,x)u_{0}(y)dy\right|\leq C_{\alpha} \lim_{t \rightarrow \infty}t^{\frac{q}{2}-\frac{1}{\alpha-1}}=0\,.
    \end{equation}
    The contradiction will be obtained by showing that the above limit is always bounded away from zero for any $u_{0}$, and might even diverge if $u_0 \notin L^1(\Rn)$. We have
    \begin{equation*}
\begin{split}
     t^{\frac{q}{2}}\left|\int_{\Rn}\Gamma(0,y;t,x)u_{0}(y)dy\right|&\stackrel{\eqref{est2}}\geq t^{\frac{q}{2}}\int_{\Rn}\frac{\exp\left(-\frac{\rho d^{2}_{X}(x,y)}{t}\right)}{\rho|B_{X}(x,\sqrt{t})|}u_0(y)dy\\&
  \geq C\frac{t^{\frac{q}{2}}}{\sum\limits_{j=m}^{q}f_{j}(x)t^{\frac{j}{2}}}\int_{\Rn} \exp\left(-\frac{\rho d^{2}_{X}(x,y)}{t}\right)u_{0}(y)dy,
\end{split}
\end{equation*}
which gives 
\begin{equation}\label{asd2}
    \lim_{t \rightarrow \infty} t^{\frac{q}{2}}\left|\int_{\Rn}\Gamma(0,y;t,x)u_{0}(y)dy\right|\geq C \|u_0\|_{L^1(\Rn)}\,,
\end{equation}
    and we have proved part (i) of Theorem \ref{thmnon}.
    
Let us now consider the critical case $\alpha=1+\frac{2}{q}$ for \eqref{non2}. Similarly, suppose that there exists a global solution of \eqref{non2}. We redefine $u$ on a null set so that \eqref{non2}  holds everywhere in $(0, \infty) \times \Rn$ and obtain
\begin{multline}\label{ca3}
u(t+t_{0},x)=\int_{\Rn}\Gamma(0,y;t+t_0,x)u(t_{0},y)dy\\
+\int_{0}^{t+t_0}\int_{\Rn}\Gamma(0,y;t+t_0-\tau,x)f(u(\tau,y))dyd\tau,\quad t_{0}>0.
\end{multline} 
For $u(t,x)$ satisfying \eqref{non2} using Corollary \ref{cor2} and $t>1$, we have
\begin{equation*}
\begin{split}
   \int_{\mathbb{R}^{n}} \exp\left(-\frac{\rho d^{2}_{X}(0,y)}{t}\right)u(\tau,y)dy&\stackrel{\eqref{est2}}\leq \rho|B_{X}(0,\sqrt{t})|\int_{\mathbb{R}^{n}} \Ga(0,x;t,y)u(\tau,y)dy\\&
  \leq \rho  t^{\frac{q}{2}}|B_{X}(0,1)|\int_{\mathbb{R}^{n}} \Ga(0,x;t,y)u(\tau,y)dy\\&
   \leq C\rho  t^{\frac{1}{\alpha-1}}\left\| \int_{\mathbb{R}^{n}} \Ga(0,\cdot;t,y)u(\tau,y)dy\right\|_{L^{\infty}(\Rn)}\\&
   \stackrel{\eqref{cor.3.5.est}}\leq C'\,,
\end{split}
\end{equation*}
and allowing $t \rightarrow \infty$ the above yields
\begin{equation}\label{thm.contr}
    \|u(0,\cdot)\|_{L^1(\Rn)}\leq C'\,,\quad \hbox{for a.e. $\tau\geq 0$}.
\end{equation}
Since $u(t,x)$ satisfies \eqref{non2} with the use of \eqref{est2} we obtain
\begin{equation*}
    \begin{split}
        u(t,x)&\geq \int_{\Rn}\Gamma(0,y;t,x)u_{0}(y)dy\\&
        \geq \frac{1}{\rho|B_{X}(x,\sqrt{t})|}\int_{\Rn}\exp\left(-\frac{\rho d^{2}_{X}(x,y)}{t}\right)u_{0}(y)dy\\&
        \geq\frac{\exp\left(-\frac{2\rho d^{2}_{X}(x,0)}{t}\right)}{\rho|B_{X}(x,\sqrt{t})|} \int_{\Rn}\exp\left(-\frac{2\rho d^{2}_{X}(y,0)}{t}\right)u_{0}(y)dy,
    \end{split}
\end{equation*}
where for the last estimate we have used the fact that the CC-distance satisfies the triangle inequality.
The latter implies that for all $x \in \Rn$ we have the estimate
\begin{equation*}
    \begin{split}
        u(2\rho^{2},x)&\geq \frac{\exp\left(-\frac{ d^{2}_{X}(x,0)}{\rho}\right)}{\rho|B_{X}(x,\sqrt{2}\rho)|} \int_{\Rn}\exp\left(-\frac{d^{2}_{X}(y,0)}{\rho}\right)u_{0}(y)dy
        \end{split}
\end{equation*}
        \begin{equation*}
    \begin{split}
       &\stackrel{\eqref{est2}}\geq \frac{1}{\rho|B_{X}(x,\sqrt{2}\rho)|}\frac{\Ga(0,0;1,x) |B_X(x,1)|}{\rho}\int_{\Rn}\exp\left(-\frac{d^{2}_{X}(y,0)}{\rho}\right)u_{0}(y)dy\\&
        \stackrel{\eqref{prop.ball}}\geq \frac{\Ga(0,0;1,x)}{\rho^2}\gamma_1 \left( \frac{1}{\sqrt{2}\rho}\right)^n \int_{\Rn}\exp\left(-\frac{ d^{2}_{X}(y,0)}{\rho}\right)u_{0}(y)dy\\&
        =C\Ga(0,0;1,x)\,,
    \end{split}
\end{equation*}
since $\int_{\Rn}\exp\left(-\frac{\rho d^{2}_{X}(y,0)}{\rho}\right)u_{0}(y)dy< \infty$ due to the fact that by the above $u_0 \in L^1(\Rn)$. 
The last fact can be shown from \eqref{asd1}, \eqref{asd2} and $\alpha=1+\frac{2}{q}$, and  we have $u_0 \in L^1(\Rn)$. 
Consequently by \eqref{f} and \eqref{ca3}, we have 
\begin{equation*}
        u(s+2\rho^{2},x)\geq \int_{\Rn}\Gamma(0,y;s,x)u(2\rho^{2},y)dy
\end{equation*}
        \begin{equation}\label{ocen4}
    \begin{split}
&\geq C\int_{\Rn}\Gamma(0,y;s,x)\Ga(0,0;1,y)dy\\&
=C\Ga(0,0;s+1,x),\quad \hbox{for any $s\geq 0$, $x \in \Rn$.}
    \end{split}
\end{equation}
For all $s>0$ we have
\begin{equation}\label{ocen3}
    \begin{split}
        &\|\Ga^{\alpha}(0,0;s+1,\cdot)\|_{L^{1}(\Rn)}=\int_{\Rn}\Ga^{\alpha}(0,0;s+1,y)dy\\&
       \stackrel{\eqref{est2}}\geq \frac{1}{\rho^{\alpha}|B_{X}(0,\sqrt{s+1})|^{\alpha}} \int_{\Rn}\exp\left(-\frac{\alpha\rho d^{2}_{X}(0,y)}{s+1}\right)dy\\&
       =\frac{|B_{X}(0,\sqrt{s+1})|}{\rho^{\alpha+1}|B_{X}(0,\sqrt{s+1})|^{\alpha}}\left[\frac{\rho}{|B_{X}(0,\sqrt{s+1})|}\int_{\Rn}\exp\left(-\frac{d^{2}_{X}(0,y)}{\rho\left(\frac{s+1}{\alpha \rho^{2}}\right)}\right)dy\right]\\&
       \stackrel{\eqref{est2}}\geq \frac{|B_{X}(0,\sqrt{s+1})|}{\rho^{\alpha+1}|B_{X}(0,\sqrt{s+1})|^{\alpha}}\frac{|B_{X}(0,\sqrt{\frac{s+1}{\alpha\rho^2}})|}{|B_{X}(0,\sqrt{s+1})|}\int_{\Rn}\Ga\left(0,0;\frac{s+1}{\alpha \rho^{2}},y\right)dy\\&
       \stackrel{\eqref{prop.ball}}\geq \frac{|B_{X}(0,\sqrt{s+1})|}{\rho^{\alpha+1}|B_{X}(0,\sqrt{s+1})|^{\alpha}} C\int_{\Rn}\Ga\left(0,0;\frac{s+1}{\alpha \rho^{2}},y\right)dy \\&
       \stackrel{\eqref{d}}=C\frac{|B_{X}(0,\sqrt{s+1})|}{\rho^{\alpha+1}|B_{X}(0,\sqrt{s+1})|^{\alpha}}\\&
       = C \frac{|B_{X}(0,\sqrt{s+1})|^{1-\alpha}}{\rho^{\alpha+1}}\\&
       \stackrel{\eqref{propcc}}\geq C (s+1)^{\frac{q(1-\alpha)}{2}} \\&
       =C(s+1)^{-1}\,,
    \end{split}
\end{equation}
where for the last inequality we used the relation $\alpha=1+\frac{2}{q}$.
Using Fubini's theorem and \eqref{d}, for fixed $t>0$ we have
\begin{equation}\label{ocen5}
\begin{split}
\left\|\int_{\Rn}\Ga(0,\cdot;t,y)\varphi(y)dy\right\|_{L^{1}(\Rn)}&=\int_{\Rn}\int_{\Rn}\Ga(0,x;t,y)\varphi(y)dydx\\&
=\int_{\Rn}\varphi(y)\left[\int_{\Rn}\Ga(0,y;t,x)dx\right]dy\\&
=\int_{\Rn}\varphi(y)dy\,.
\end{split}
\end{equation}
Before moving on to the final estimate contradicting our assumption, we first need to make the following auxiliary estimate
\begin{equation}\label{EQ:aux}
\begin{split}
    u(t+2\rho^2,x)&\stackrel{\eqref{ca3}}\geq \int_{0}^{t+2\rho^2}\int_{\Rn}\Ga(0,y;t+2\rho^2-\tau,x)f(u(\tau,y))dyd\tau\\&
    \stackrel{s=\tau+2\rho^2}=\int_{-2\rho^2}^{t}\int_{\Rn}\Ga(0,y;t-s,x)f(u(s+2\rho^2,y))dyds\\&
    \stackrel{\eqref{EQ:s<t}}=\int_{0}^{t}\int_{\Rn}\Ga(0,y;t-s,x)f(u(s+2\rho^2,y))dyds\,.
    \end{split}
\end{equation}
Now by  \eqref{ocen5} for $\phi(y)=\Gamma^{\alpha}(0,0;\tau+1,y)$ we obtain
\begin{equation}
    \label{EQ:Aux2}
    \int_{\Rn}\int_{\Rn}\Gamma(0,y;t-\tau,x)\Gamma^{\alpha}(0,0;\tau+1,y)dxdy=\int_{\Rn}\Gamma^{\alpha}(0,0;\tau+1,y)dy\,.
\end{equation}
We have 
\begin{equation*}
    \begin{split}
   \|u(t+2\rho^{2},\cdot)\|_{L^{1}(\Rn)}&\stackrel{\eqref{EQ:aux}}\geq\int_{0}^{t}\int_{\Rn}\int_{\Rn}\Gamma(0,y;t-\tau,x)f(u(\tau+2\rho^{2},y))dxdyd\tau\\&
   \stackrel{f(u)\geq Bu^\alpha}\geq B\int_{0}^{t}\int_{\Rn}\int_{\Rn}\Gamma(0,y;t-\tau,x)u^{\alpha}(\tau+2\rho^{2},y)dxdyd\tau\\&
   \stackrel{\eqref{ocen4}}\geq CB\int_{0}^{t}\int_{\Rn}\int_{\Rn}\Gamma(0,y;t-\tau,x)\Gamma^{\alpha}(0,0;\tau+1,y)dxdyd\tau\\&
\stackrel{\eqref{EQ:Aux2}}=CB\int_{0}^{t}\int_{\Rn}\Gamma^{\alpha}(0,0;\tau+1,y)dyd\tau\\&
   \stackrel{\eqref{ocen3}}\geq C\int_{0}^{t}(\tau+1)^{-1}d\tau\,,
     \end{split}
\end{equation*}
and we have shown that for all $t>0$ we have 
\begin{equation}\label{contra}
     \|u(t+2\rho^{2},\cdot)\|_{L^{1}(\Rn)}\geq C \log(t+1)\,.
\end{equation}

Allowing $t\rightarrow +\infty$ in \eqref{contra} and in \eqref{thm.contr} we get a contradiction.
Therefore,  the global solution to \eqref{non2} does not exist. Summarising, the proof of Theorem \ref{thmnon} is now complete.
\end{proof}
\begin{proof}[Proof of Lemma \ref{lem1}]
In several steps of this proof, we use Fubini's theorem and the symmetric property \eqref{sym2} without further explanations. For  $0<\tau<t$, \eqref{f} and any measurable function $G:\Rn\rightarrow[0,+\infty]$, we have 
 \begin{equation}\label{prlem1}
     \begin{split}  
     \int_{\Rn}\Ga(0,x;t,y)G(y)dy&
     \stackrel{\eqref{f}}=\int_{\Rn}\int_{\Rn}\Ga(0,w;t-\tau,x)\Ga(0,y;\tau,w)G(y)dwdy\\& 
     =\int_{\Rn}\Ga(0,w;t-\tau,x)\int_{\Rn}\Ga(0,y;\tau,w)G(y)dydw\,.
     \end{split}
 \end{equation}
 Now, for the same $G$, using the H\"{o}lder inequality with $\frac{1}{\alpha}+\frac{1}{\alpha'}=1$ we have the identity
 \begin{equation}\label{prlem2}
     \begin{split}
\left(\int_{\Rn}\Ga(0,x;t,y)G(y)dy\right)^{\alpha}&=\left(\int_{\Rn}\Ga^{\frac{1}{\alpha'}}(0,x;t,y)\Ga^{\frac{1}{\alpha}}(0,x;t,y)G(y)dy\right)^{\alpha}\\&
\leq \left(\int_{\Rn}\Ga(0,x;t,y)dy\right)^{\frac{\alpha}{\alpha'}} \left(\int_{\Rn}\Ga(0,x;t,y)G^{\alpha}(y)dy\right)\\&
\stackrel{\eqref{d}}=\left(\int_{\Rn}\Ga(0,x;t,y)G^{\alpha}(y)dy\right).
     \end{split}
 \end{equation}
 Now, let us redefine $u$ on a null set and suppose that \eqref{lemin} holds everywhere in $[0, T] \times \Rn$. Fix $\tau \in [0,T]$ and denote $M_{\tau} :=\{x\in\Rn:v(\tau,x)<\infty\}$. For every $t \in [0,\tau]$, we have
 \begin{equation}\label{prlem3}
     \begin{split}
        & \int_{\Rn}\Ga(0,x;\tau-t,w)v(t,w)dw\stackrel{\eqref{lemin}}\geq\int_{\Rn}\int_{\Rn}\Ga(0,x;\tau-t,w)\Gamma(0,y;t,w)v_{0}(y)dydw\\&
  +\int_{0}^{t}\int_{\Rn}\Ga(0,x;\tau-t,w)\int_{\Rn}\Gamma(0,y;t-s,w)f(v(s,y))dyds dw\\&
\stackrel{\eqref{f},\eqref{prlem1}}=\int_{\Rn}\Ga(0,x;\tau,w)v_{0}(y)dy+\int_{0}^{t}\int_{\Rn}\Ga(0,x;\tau-s,y)f(v(s,y))dyds\\&
\stackrel{f(u)\geq Bu^\alpha}\geq\int_{\Rn}\Ga(0,x;\tau,y)v_{0}(y)dy+B\int_{0}^{t}\int_{\Rn}\Ga(0,x;\tau-s,y)v^{\alpha}(s,y)dyds\\&
:=g(t,x),
     \end{split}
 \end{equation}
where \eqref{prlem1} was used for $G=G(\cdot)=f(v(s,\cdot))$.

 By \eqref{lemin} and the above we have 
 \begin{equation*}
 \begin{split}
     g(\tau,x)&= \int_{\Rn}\Ga(0,x;\tau,y)v_{0}(y)dy+B\int_{0}^{t}\int_{\Rn}\Ga(0,x;\tau-s,y)v^{\alpha}(s,y)dyds\\&
     \leq v(\tau,x),
 \end{split}
 \end{equation*}
where $g(t,x)<\infty$  for a.e. $(t,x)\in [0,\tau]\times M_{\tau}$. For a fixed $x\in M_{\tau}$ the function  $\varsigma (t):=g(t,x)$ is absolutely continuous on $[0,\tau]$. (Indeed $\varsigma(t)=\varsigma(0)+\int_{0}^{t}h(s)ds$, where the function $h \in L^{1}([0,\tau])$ is the one arising from the definition of $g$.) Hence $\varsigma$ is differentiable almost everywhere in $[0,\tau]$ with derivative
 \begin{equation}\label{prlem4}
\frac{d\varsigma(t)}{dt}\stackrel{\eqref{prlem2}}\geq B\left[\int_{\Rn}\Ga(0,x;\tau-t,w)v      (t,w)dw\right]^{\alpha}
\stackrel{\eqref{prlem3}}\geq B\varsigma^{\alpha}(t),\,\,\,\text{ for}\, t\in[0,\tau].
 \end{equation}
 Now since $\varsigma(t)>0$ and $1-\alpha<0$ we obtain
 \[
 \frac{d\varsigma^{1-\alpha}(t)}{dt}\stackrel{\eqref{prlem4}}\leq -B(\alpha-1)\,,
 \]
 and integrating the latter over $[0,\tau]$  we get 
 \begin{equation*}
     \left[\int_{\Rn}\Ga(0,x;\tau,w)v_{0}(w)dw\right]^{1-\alpha}=\varsigma^{1-\alpha}(0)\geq \varsigma^{1-\alpha}(\tau)+B(\alpha-1)\tau\geq B(\alpha-1)\tau,
 \end{equation*}
 that is,
 \begin{equation*}
     \tau^{\frac{1}{\alpha-1}}\left\|\int_{\Rn}\Ga(0,\cdot;\tau,w)v_{0}(w)dw\right\|_{L^{\infty}(\Rn)}\leq (B(\alpha-1))^{-\frac{1}{\alpha-1}}.
 \end{equation*}
 Since for $v_0 \in L^{\infty}(\Rn)$ the function 
 \begin{equation*}
     t\mapsto t^{\frac{1}{\alpha-1}}\left\|\int_{\Rn}\Ga(0,\cdot;t,w)v_{0}(w)dw\right\|_{L^{\infty}(\Rn)}
 \end{equation*}
 is continuous in  $t \in [0,T]$ and $\tau \in [0,T]$ was arbitrary, the latter inequality holds true for a.e. $t \in [0,T]$, and the proof is complete.
 \end{proof}

\section{Time-dependent case}
In this section we consider the Cauchy problem \eqref{intro.prob} adding a time-dependent function $\varphi$ in front of the non-linearity $f(u)$, that is we consider the non-linear heat equation 
\begin{equation}
    \label{prob2}
    \begin{cases}
         u_{t}(t,x)-\L u(t,x)= \varphi(t)f(u(t,x)),\,\,\,\,\,(t,x)\in \mathbb{R}_{+}\times \Rn,\\
          u(0,x)=u_{0}(x),\,\,\,\,\,x\in \Rn,
    \end{cases}
\end{equation}
where $\L=\sum\limits_{i=1}^{m}X_{i}^{2}$ is a H\"{o}rmander sum of squares as before.  

The mild solution to the Cauchy problem \eqref{prob2} is given by 
\begin{equation}\label{solutionto-t}
 u(t,x)= \int_{\Rn}\Gamma(0,y;t,x)u_{0}(y)dy+\int_{0}^{t}\int_{\Rn}\Gamma(0,y;t-\tau,x)\varphi(\tau)f(u(\tau,y))dyd\tau\,.
\end{equation}

Before stating the assumptions on the involved function $f$ and $\varphi$, we need to define:
\begin{defn}
We define the  \emph{majorant} function \( f_M : [0,+\infty) \to [0,+\infty) \) associated with a given function \( f \) by
\begin{equation}
    f_M(v) := \sup_{\alpha \in (0,1)} \frac{f(\alpha v)}{f(\alpha)}, \qquad v \geq 0.
\end{equation}
\end{defn}
An immediate property of $f_M$ is then
\begin{equation}\label{maj1}
 f(\alpha v)\leq f(\alpha)f_{M}(v),\,\,\,\alpha\in (0,1),\,\,\,\,v\geq 0.    
\end{equation}

\begin{assumptions*}\label{assumption,f,phi}
We assume that $f$ is a  continuous, non-negative function with $f(0)=0$ and $f(v)>0$ when $v>0$, and that the mapping  $v\mapsto\frac{f(v)}{v}$ is non-decreasing. Moreover, we assume that  the associated majorant function $f_M$ satisfies\begin{equation}
    \label{maj}
    \lim_{v \rightarrow 0^{+}}\frac{f_M(v)}{v}=0\,.
\end{equation} For the function $\varphi$ we just assume that $0\leq\varphi\in L^{1}_{\text{loc}}[0,+\infty)$. 
\end{assumptions*}
For $f$ and $\varphi$ satisfying the assumptions above, we have the following theorem: 
\begin{thm}
    \label{thm.t-depen}
      Let $X=\{X_{1},\ldots,X_{m}\}$ be a system of vector fields on $\Rn$ that satisfies Assumption \ref{as1}.  If \eqref{prob2} does not have a global solution for any $0<u_0 \in  L^{1}(\Rn)$, then
      \begin{enumerate}[label=(\roman*)]
    \item for all $0<w \in  L^{1}(\Rn)$, we have:
    \[
    \int_{0}^{\infty}\varphi(\tau) \frac{f\left( \|\int_{\Rn}\Gamma(0,y;t,\cdot)w(y)dy \|_{L^{\infty}(\Rn)}\right)}{\|\int_{\Rn} \Gamma(0,y;t,\cdot)w(y)\|_{L^{\infty}(\Rn)}} d\tau=+\infty\,,
    \]
    \item and for every $\omega>0$ we have 
    \[\int_{1}^{\infty}\varphi(\tau)\tau^{\frac{q}{2}} f(\omega \tau^{-\frac{q}{2}})d\tau=+\infty\,,
    \]
   where $q$ is the homogeneous dimension given by \eqref{homdim}.
\end{enumerate}

\end{thm}

\begin{proof}
    To reach a contradiction, assume that there exists $w>0$ such that
   \begin{equation}
       \label{Z}
       Z:=\int_{0}^{\infty}\varphi(\tau) \frac{f\left( \|\int_{\Rn}\Gamma(0,y;t,\cdot)w(y)dy \|_{L^{\infty}(\Rn)}\right)}{\|\int_{\Rn} \Gamma(0,y;t,\cdot)w(y)\|_{L^{\infty}(\Rn)}} d\tau<\infty
   \end{equation}
   where $w>0$ is a positive function such that $w\in L^{1}(\Rn)$. A contradiction arises by showing that \eqref{prob2} admits a global solution for some $0<u_{0}\in  L^{1}(\Rn)$.

    We define $u_0:=\lambda w$ and  also choose $\lambda    \in (0,1)$ such that 
\begin{equation}\label{ocenz2}
   \sqrt{\lambda} \left\| \int_{\Rn}\Gamma(0,y;t,\cdot)w(y)dy \right\|_{L^{\infty}(\Rn)}< 1\,.
\end{equation}
Note that condition \eqref{maj} implies that for each $\epsilon>0$, there exists $\delta>0$ such that if $0<\sqrt{\lambda}(1+Z)<\delta$, then  
$\frac{f_{M} (\sqrt{\lambda} (1+Z))}{\sqrt{\lambda} (1+Z)}<\epsilon,$ so that
\begin{equation}\label{epsilon}
0\leq \frac{f_{M} (\sqrt{\lambda} (1+Z))}{\lambda}<(1+Z)\frac{\epsilon}{\sqrt{\lambda}}\,.\end{equation}
 Let us now define the auxiliary sequence of positive functions $\{v_\zeta\}_{\zeta\geq0}$ on $\mathbb{R}_{+}\times \Rn$ by 
\begin{equation*}
  v_{\zeta}(t,x):=
  \begin{cases}
    \displaystyle \int_{\mathbb{R}^n} \Gamma(0,y;t,x)\, u_{0}(y)\, dy, 
    & \text{if } \zeta = 0, \\
    \begin{aligned}
      &\displaystyle \int_{\mathbb{R}^n} \Gamma(0,y;t,x)\, u_{0}(y)\, dy \\
      &\quad  +\int_{0}^{t}\int_{\mathbb{R}^n} \Gamma(0,y;t-\tau,x)\, \varphi(\tau)\, f(v_{\zeta-1}(\tau,y))\, dy\, d\tau,
    \end{aligned}
    & \text{if } \zeta \geq 1.
  \end{cases}
\end{equation*}
We proceed by induction to show that, for each \( \zeta \geq 1 \) and all \( (t,x) \in \mathbb{R}_+ \times \mathbb{R}^n \), the following property is satisfied:

\begin{equation}
    \label{vk<vo}
    v_\zeta(t,x) \leq (1+Z) \int_{\Rn}\Gamma(0,y;t,x)u_{0}(y)dy\,.
\end{equation}
Inequality \eqref{vk<vo} is immediate when $k=0$ since we have 
\[v_{0}(t,x)= \int_{\Rn}\Gamma(0,y;t,x)u_{0}(y)dy\leq (1+Z)\int_{\Rn}\Gamma(0,y;t,x)u_{0}(y)dy\]for all $(t,x)\in \mathbb{R}_{+}\times \Rn$. 
Let $\zeta \in \mathbb{N}$, and assume that
\begin{equation}\label{vzeta1}
    v_{\zeta}(t,x) \leq (1+Z)\int_{\Rn}\Gamma(0,y;t,x)u_{0}(y)dy.
\end{equation}
Using the inductive hypothesis we have
\[
\begin{aligned}
v_{\zeta+1}(t,x) &\leq \int_{\mathbb{R}^n} \Gamma(0,y;t,x) u_0(y)  dy \\
&\quad + \int_0^t \int_{\mathbb{R}^n} \Gamma(0,y;t-\tau,x) \varphi(\tau) f\left((1+Z) \int_{\mathbb{R}^n} \Gamma(0,z;\tau,y) u_0(z)  dz\right) dy  d\tau.
\end{aligned}
\]

Let
\[
A_{u_0}(t;y):=\int_{\Rn}\Gamma(0,z;t,y)u_0(z)\,dz.
\]

Using the fact that $\frac{f(v)}{v}$ is a non-decreasing function, and noting that $A_{u_0}(\tau,y) > 0$ (since $u_0 > 0$ and the heat kernel is positive), we get

\[
\begin{aligned}
v_{\zeta+1}(t,x) &\leq A_{u_0}(t,x) + \int_0^t \int_{\mathbb{R}^n} \Gamma(0,y;t-\tau,x) \varphi(\tau) \left[ A_{u_0}(\tau,y) \frac{f((1+Z) A_{u_0}(\tau,y))}{A_{u_0}(\tau,y)} \right] dy  d\tau \\
&\leq A_{u_0}(t,x) + \int_0^t \int_{\mathbb{R}^n} \Gamma(0,y;t-\tau,x) \varphi(\tau) \left[ A_{u_0}(\tau,y) \frac{f((1+Z) \|A_{u_0}(\tau,\cdot)\|_{L^\infty})}{\|A_{u_0}(\tau,\cdot)\|_{L^\infty}} \right] dy  d\tau \\
&= A_{u_0}(t,x) + \frac{f((1+Z) \|A_{u_0}(\tau,\cdot)\|_{L^\infty})}{\|A_{u_0}(\tau,\cdot)\|_{L^\infty}} \int_0^t \varphi(\tau) \left[ \int_{\mathbb{R}^n} \Gamma(0,y;t-\tau,x) A_{u_0}(\tau,y)  dy \right] d\tau.
\end{aligned}
\]

Note that by the reproduction formula \eqref{f} we have 

\[
\int_{\mathbb{R}^n} \Gamma(0,y;t-\tau,x) A_{u_0}(\tau,y)  dy = A_{u_0}(t,x)\,,
\]
so that 
\[
\begin{aligned}
v_{\zeta+1}(t,x) &\leq A_{u_0}(t,x) + A_{u_0}(t,x) \int_0^t \varphi(\tau) \frac{f((1+Z) \|A_{u_0}(\tau,\cdot)\|_{L^{\infty}(\Rn)})}{\|A_{u_0}(\tau,\cdot)\|_{L^{\infty}(\Rn)} d\tau} \\
&= A_{u_0}(t,x) \left[ 1 + \int_0^t \varphi(\tau) \frac{f((1+Z) \|A_{u_0}(\tau,\cdot)\|_{L^{\infty}(\Rn)})}{\|A_{u_0}(\tau,\cdot)\|_{L^{\infty}(\Rn)}} d\tau \right].
\end{aligned}
\]
Hence for $u_0(x) = \lambda w(x)$, with $1>\lambda > 0$ as in \eqref{ocenz2}, by the definition of the majorant function and the non-decreasing property of the function $\frac{f(v)}{v}$ one obtains

\[
\begin{aligned}
v_{\zeta+1}(t,x) &\leq A_{u_0}(t,x) + A_{u_0}(t,x) \int_0^t \varphi(\tau) \frac{f(\lambda(1+Z)\|A_w(\tau,\cdot)\|_{L^{\infty}(\Rn)})}{\lambda\|A_w(\tau,\cdot)\|_{L^{\infty}(\Rn)})} d\tau \\
&= A_{u_0}(t,x) \left[ 1 + \int_0^t \varphi(\tau) \frac{f(\lambda(1+Z)\|A_w(\tau,\cdot)\|_{L^{\infty}(\Rn)})}{\lambda\|A_w(\tau,\cdot)\|_{L^\infty(\Rn)}} d\tau \right].
\end{aligned}
\]

Now using the majorant function property \eqref{maj1} for $\alpha \in (0,1)$, with $\alpha = \sqrt{\lambda}$ and $v = \sqrt{\lambda}(1+Z)\|A_w(\tau,\cdot)\|_{L^\infty(\Rn)}$ we can estimate $v_{\zeta+1}$ further as

\[
\begin{aligned}
v_{\zeta+1} &\leq  A_{u_0}(t,x) + \frac{f_M(\sqrt{\lambda}(1+Z))}{\sqrt{\lambda}} A_{u_0}(t,x) \int_0^t \varphi(\tau) \frac{f(\sqrt{\lambda}\|A_w(\tau,\cdot)\|_{L^{\infty}(\Rn)})}{\sqrt{\lambda}\|A_w(\tau,\cdot)\|_{L^{\infty}(\Rn)}} d\tau \\
&\leq A_{u_0}(t,x) + \frac{\sqrt{\lambda}f_M(\sqrt{\lambda}(1+Z))}{\lambda} A_{u_0}(t,x) \int_0^\infty \varphi(\tau) \frac{f(\|A_w(\tau,\cdot)\|_{L^{\infty}(\Rn)})}{\|A_w(\tau,\cdot)\|_{L^{\infty}(\Rn)}} d\tau\,,
\end{aligned}
\]
so that, by using \eqref{epsilon} and recalling the definition of $Z$, we obtain 

\[
\begin{aligned}
v_{\zeta+1} &\leq B_{u_0}(t,x) + \frac{\sqrt{\lambda} \cdot \epsilon \sqrt{\lambda} (1+Z)}{\lambda} B_{u_0}(t,x) Z \\
&= B_{u_0}(t,x) + \epsilon (1+Z) Z B_{u_0}(t,x) \\
&= B_{u_0}(t,x) \left[ 1 + \epsilon (1+Z) Z \right].
\end{aligned}
\]

Finally, choosing $\epsilon$ such that $\epsilon (1+Z)  < 1$, we  get 

\[
v_{\zeta+1}(t,x) \leq (1+Z) B_{u_0}(t,x),
\]
which completes the proof of \eqref{vk<vo}.

Next, we will show, again by mathematical induction, that $v_{\zeta}\leq v_{\zeta+1}$ for all $\zeta \geq 0$. To this end, we note that the positivity $f$ and $\varphi$ give \[v_{0}(t,x)=B_{u_0}(t,x)\leq B_{u_0}(t,x)+\int_{0}^{t}\int_{\Rn}\Gamma(0,y;t-\tau,x)\varphi(\tau)f(v_{0}(\tau,y))dyd\tau=v_{1}(t,x)\,.\]   Assume now that for some fixed $\zeta\in\mathbb{N}$ the inequality $v_{\zeta}(t,x)\geq v_{\zeta-1}(t,x)$  holds true. Then, since $f$ is a non-decreasing function, it is immediate to see that we also have
\[
v_{\zeta+1}(t,x) \geq v_{\zeta}(t,x)\,,
\]
implying that $\{v_\zeta\}$ is a non-decreasing sequence with respect to $\zeta$.

The monotonicity of $\{v_\zeta\}$, together with the (uniform) upper bound for each $v_\zeta$ given by \eqref{vk<vo}, imply that the limit $\lim\limits\limits_{\zeta \rightarrow \infty}v_\zeta(t,x)$ exists globally. Additionally, by the continuity of $f$ together with the definition of $v_\zeta$ we get 
\[
\lim_{\zeta \rightarrow \infty}v_\zeta(t,x)=\int_{\Rn}\Gamma(0,y;t,x)u_{0}(y)dy+\int_{0}^{t}\int_{\Rn}\Gamma(0,y;t-\tau,x)\varphi(\tau)f(\lim_{\zeta \rightarrow \infty}v_{\zeta}(\tau,y))dyd\tau.
\]
Now since by \eqref{solutionto-t} $\lim_{\zeta \rightarrow \infty}v_\zeta(t,x)=u(t,x)$, the upper bound \eqref{vk<vo} implies that 
\begin{equation}\label{upper.bound.u}
u(t,x)\leq (1+Z) \int_{\Rn}\Gamma(0,y;t,x)u_{0}(y)dy\,,
    \end{equation}
 i.e., the solution $u(t,x)$ exists globally, and we have proved (i). 

 To establish (ii), consider any function \( u_0 \in L^1(\Rn) \) with \( u_0 > 0 \). It follows from (i) that for $B$ as above 
 \begin{equation}
     \label{a-b1}
     +\infty = \int_{0}^{1}\varphi(\tau)\frac{f(\|A_{u_0}(\tau,x)\|_{L^{\infty}(\Rn)})}{\|A_{u_0}(\tau,x)\|_{L^{\infty}(\Rn)}}\,d\tau+\int_{1}^{\infty}\varphi(\tau)\frac{f(\|A_{u_0}(\tau,x)\|_{L^{\infty}(\Rn)})}{\|A_{u_0}(\tau,x)\|_{L^{\infty}(\Rn)}}\,d\tau\,.
 \end{equation}
Observe that by the definition of $A_{u_0}(t,x)$ and the property \eqref{d} of the heat kernel, we derive the semigroup contraction 
\[
\|A_{u_0}(t,\cdot)\|_{L^{\infty}(\Rn)}\leq \|u_0\|_{L^{\infty}(\Rn)}\,.
\]
Combining this fact with the continuity and positivity of both \( f \) and \( \varphi \), and using that the function \( v \mapsto \frac{f(v)}{v} \) is non-decreasing, we obtain
\[
\frac{f\left(\|A_{u_0}(\tau,\cdot)\|_{L^{\infty}(\mathbb{R}^n)}\right)}{\|A_{u_0}(\tau,\cdot)\|_{L^{\infty}(\mathbb{R}^n)}} 
\leq \frac{f\left(\|u_0\|_{L^{\infty}(\Rn)}\right)}{\|u_0\|_{L^{\infty}(\Rn)}}.
\]
Therefore,
\[
\int_0^1 \varphi(\tau)\, \frac{f\left(\|A_{u_0}(\tau,\cdot)\|_{L^{\infty}(\mathbb{R}^n)}\right)}{\|A_{u_0}(\tau,\cdot)\|_{L^{\infty}(\mathbb{R}^n)}}\, d\tau 
\leq \frac{f\left(\|u_0\|_{L^{\infty}(\Rn)}\right)}{\|u_0\|_{L^{\infty}(\Rn)}} \int_0^1 \varphi(\tau)\, d\tau < +\infty\,,
\]
and by \eqref{a-b1} we get 
\begin{eqnarray}
\label{a-b41}
+\infty & = & \int_{1}^{\infty} \varphi(\tau)\, 
\frac{f\left(\|A_{u_0}(\tau,\cdot)\|_{L^{\infty}(\mathbb{R}^n)}\right)}{\|A_{u_0}(\tau,\cdot)\|_{L^{\infty}(\mathbb{R}^n)}}\, d\tau.
\end{eqnarray}
Using the upper bound for the heat kernel by \eqref{est2}, together with the control volume \eqref{ocencc1}, we obtain for $\rho>0$ and for all $t>0$

\begin{equation}\label{semigroup.est1}
    \begin{split}
          A_{u_0}(t,x) &\leq  \frac{\rho}{|B_X(x,\sqrt{t})|} \int_{\Rn} \exp\left( -\frac{d_X^2(x,y)}{\rho t} \right) u_0(y)\, dy\\&
        \leq \frac{\rho}{C t^{q/2}} \int_{\Rn} \exp\left( -\frac{d_X^2(x,y)}{\rho t} \right) u_0(y)\, dy.
    \end{split}
    \end{equation}
    Consequently, we have 
    \[
   \| A_{u_0}(t,\cdot)\|_{L^{\infty}(\Rn)} \leq \omega t^{-q/2}\,,
    \]
for $q$ as in \eqref{homdim}, where we have defined $\omega:= C^{-1}\rho \int_{\Rn}u_0(y)\,dy$.

Hence, since \( 0 \leq \varphi \), and using the monotonicity of the function \( \frac{f(v)}{v} \), we obtain
\begin{equation}
\label{eq:divergence}
+\infty = \int_{1}^{\infty} \varphi(\tau)\, 
\frac{f\left(\|A_{u_0}(\tau,\cdot)\|_{L^{\infty}(\mathbb{R}^n)}\right)}{\|A_{u_0}(\tau,\cdot))\|_{L^{\infty}(\mathbb{R}^n)}}\, d\tau
\leq \int_{1}^{\infty} \varphi(\tau)\, \frac{f(\omega \tau^{-\frac{q}{2}})}{\omega \tau^{-\frac{q}{2}}}\, d\tau,
\end{equation}
and the latter implies that
\[
\int_{1}^{\infty} \varphi(\tau)\, \frac{f(\omega \tau^{-\frac{q}{2}})}{\tau^{-\frac{q}{2}}}\, d\tau = +\infty.
\]
Since this holds for any \( 0 < u_0 \in L^1(\mathbb{R}^n) \), and thus for any \( \omega > 0 \), the proof is complete.

 \end{proof}

\end{document}